\documentclass[12pt, reqno]{amsart}

\usepackage{amsmath, amsthm, amssymb, enumerate, amscd}

\theoremstyle{plain}
\newtheorem{theorem}{Theorem}[section]
\newtheorem{proposition}[theorem]{Proposition}
\newtheorem{corollary}[theorem]{Corollary}

\theoremstyle{definition}
\newtheorem{definition}[theorem]{Definition}

\newtheorem{remark}[theorem]{Remark}

\newtheorem{summary}[theorem]{Summary}

\newtheorem{examples}[theorem]{Examples}

\newcommand{\F}{{\mathbb{F}}}

\newcommand{\CC}{{\mathcal{C}}}
\newcommand{\DC}{{\mathcal{D}}}
\newcommand{\PC}{{\mathcal{P}}}

\newcommand{\um}{{\mathbf{u}}}
\newcommand{\vm}{{\mathbf{v}}}
\newcommand{\wm}{{\mathbf{w}}}

\begin{document}

\title[Trunks and morphisms of neural codes] {Some remarks about trunks and morphisms of neural codes}

\author{Katie Christensen}
\address{Department of Mathematics\\ 
University of Louisville\\
Louisville, KY 40292, USA}
\email{katie.christensen@louisville.edu}
\thanks{$\dag$ the corresponding author}
\author{Hamid Kulosman$^\dag$}
\address{Department of Mathematics\\ 
University of Louisville\\
Louisville, KY 40292, USA}
\email{hamid.kulosman@louisville.edu}



\date{}

\begin{abstract} 
We give intrinsic characterizations of neural rings and homomorphisms between them. Also we introduce the notion of a basic monomial code map and characterize monomial code maps as compositions of basic monomial code maps. Finally, we characterize monomial isomorphisms between neural codes. Our work is based on the 2015 paper by C.~Curto and N.~Youngs about neural ring homomorphisms and maps between neural codes and on the 2018 paper by R.~Amzi Jeffs about morphisms of neural rings.
\end{abstract}

\maketitle

\section{Introduction}

The neural rings and ideals as an algebraic tool for analyzing the intrinsic structure of neural codes were introduced by C.~Curto et al. in 2013 in the pioneering paper \cite{civy}. In order to make our paper self-contained, we will give in this section some definitions and facts from \cite{civy}. All other notions and facts, that we assume to be known, can be found either in \cite{civy} or \cite{cy}, or in the standard references \cite{am} and \cite{cls}. The notions from the theory of categories, that we use in this paper, can be found either in \cite{a}, or in the standard reference \cite{j}.

\medskip
An element $\wm=w_1\dots w_n$ of $\F_2^n$ is called a {\it code word}, shortly {\it word}, of length $n$.
A set $\mathcal{C}\subseteq \F_2^n$ is called a {\it neural code}, shortly {\it code}, of length $n$, or on $n$ neurons.
Here $n\ge 0$. In the case $n=0$ we have only one code on $0$ neurons. We denote it by $\CC^{^{\textvisiblespace}}$. Its only element is the empty word, denoted by \textvisiblespace\,. (We assume that $\F_2^0=\{\textvisiblespace\,\}$.) We need this code so that the deletion of neurons, that we are going to consider later, is well-defined.

For a code $\CC\subseteq \F_2^n$ we define the {\it ideal of $\CC$}, $\mathcal{I}(\CC)\subseteq \F_2[X_1, \dots, X_n]$,  in the following way: 
\begin{equation*}
\mathcal{I}(\CC)=\{f\in \F_2[X_1,\dots, X_n]\;:\: f(\wm)=0 \text{ for all } \wm\in \CC\}.
\end{equation*}

Note that for any code $\CC$ in $\F_2^n$ we have $\mathcal{I}(\CC)\supseteq \mathcal{B}$, where $\mathcal{B}=(X_1^2-X_1, \dots, X_n^2-X_n)$ is the {\it Boolean ideal} of $\F_2[X_1, \dots, X_n]$. Moreover, for $\CC\subseteq \F_2^n$ we have $\mathcal{I}(\CC)=\mathcal{B}$ if and only if $\CC$ is the {\it full code} on $n$ neurons, i.e., $\CC=\F_2^n$.

For any $\wm=w_1w_2\dots w_n\in \F_2^n$ we define the {\it Lagrange polynomial} of $\wm$, $L_\wm\in \F_2[X_1,\dots, X_n]$, in the following way:
\begin{equation*}
L_\wm=\prod_{w_i=1} X_i \prod_{w_j=0} (1-X_j).
\end{equation*}
Note that $L_\wm(\wm)=1$ and $L_\wm(\vm)=0$ for any $\vm\ne \wm$ from $\F_2^n$. We have
\begin{equation*}
\mathcal{I}(\CC)=(\{L_\wm\;:\;\wm\notin\CC\}\cup\{X_i^2-X_i\;:\;i=1, \dots, n\}).
\end{equation*}

The {\it neural ring} of $\CC\subseteq \F_2^n$ is defined to be the ring
\[R_\CC=\frac{\F_2[X_1, \dots, X_n]}{\mathcal{I}(\CC)}=\F_2[x_1, \dots, x_n],\]
where $x_i=X_i+\mathcal{I}(\CC)$ for $i=1,2,\dots, n$. 
For any $f\in \F_2[X_1,\dots, X_n]$ we will denote  $\overline{f}=f+\mathcal{I}(\CC)=f(x_1,\dots, x_n)$, but we will also use $\overline{f}=\overline{f}(x_1,\dots, x_n)$. 
In particular, the image of the Lagrange polynomial $L_\wm$ is denoted by $\overline{L_\wm}$ or $\overline{L_\wm}(x_1, \dots, x_n)$.
For $A\subseteq\CC$ we denote by $L_A$ the polynomial $\sum_{\wm\in A} L_\wm$. It turns out that $R_\CC$ consists of all $\overline{L_A}$, $A\subseteq \CC$, and that they are all distinct. Moreover, if we denote by $\PC(\CC)$ the power set of $\CC$, then the bijection $R_\CC\to (\PC(\CC), \triangle, \cap)$, given by 
\[\overline{L_A}\mapsto A,\]
is a ring isomorphism. For the purpose of this paper we call the ring $(\PC(\CC),\triangle, \cap)$ as well the {\it neural ring} of $\CC$. It will always be clear from the context which version of the {\it neural ring} of $\CC$ we are using.


\section{An intrinsic characterization of neural rings and morphisms between them}

\begin{definition}[{\cite{a}}]
Let $\CC\subseteq \F_2^n$ be a code of length $n$ and $\alpha\subseteq [n]$. Then the subset of $\CC$
\[\mathrm{Tk}_\alpha^\CC=\{\wm=w_1w_2\dots w_n\in\CC\;|\;w_i=1 \text{ for all $i\in\alpha$}\}\]
is called the {\it trunk} of $\CC$ determined by $\alpha$. In particular, $\mathrm{Tk}_\emptyset^\CC=\CC$. If $|\alpha|=1$, $\mathrm{Tk}_\alpha^\CC$ is called a {\it simple trunk} of $\CC$.
\end{definition}

We will usually write $\mathrm{Tk}_i^\CC$ instead of $\mathrm{Tk}_{\{i\}}^\CC$.

\medskip
In the next theorem we give an intrinsic characterization of neural rings. The inspiration for this theorem is coming from \cite[Theorem 1.2]{cy}, where neural rings on $n$ neurons (as modules) were characterized in terms of the actions of the neural ring of the full code on $n$ neurons on them. The part of our proof in which we construct the code $\CC$ follows the proof of Theorem 1.2 from \cite{cy}.

\begin{theorem}
A non-zero commutative ring $R$ is isomorphic to the neural ring of some neural code $\CC$ if and only if there is a subset  $S=\{s_1, s_2,\dots, s_r\}$ $(r\ge 1)$ of $R$ and a sequence $T=t_1, t_2,\dots, t_n$ $(n\ge 1)$ of elements of $R$ such that the following conditions hold:

(N1) Every element $x\in R$ can be uniquely written as a sum $x=s_{j_1}+s_{j_2}+\dots+s_{j_p}$  $(p\ge 0)$ of distinct elements of $S$.

(N2) For any $t_i$ from $T$ and any $s_j\in S$, $t_is_j\in\{0,s_j\}$.

(N3) For any two distinct elements $s_j, s_k\in S$ there is at least one element $t_i$ from $T$ such that exactly one of the elements $t_is_j$, $t_is_k$ is equal to $0$.

Moreover, given a non-zero commutative ring $R$ with the properties (N1), (N2), (N3) satisfied by its subset $S$ and a sequence of its elements $T$, the code $\CC$ and the isomorphism $\phi:R\to \PC(\CC)$ can be selected in such a way that the elements of $S$ correspond to the words of $\CC$ (as singletons) and the elements of $T$ to the simple trunks of $\CC$.
\end{theorem}

\begin{proof}
Let $\CC$ be a neural code on $n$ neurons, consisting of $r$ codewords $\wm^1, \wm^2, \dots, \wm^r$, and let $(\PC(\CC), \triangle, \cap)$ be its neural ring. Let $s_j=\{\wm^j\}$ $(j=1,2,\dots, r)$, $S=\{s_1, s_2,\dots, s_r\}$, $t_i=\mathrm{Tk}_i^\CC$ $(i=1,2,\dots, n)$, $T=t_1, t_2,\dots, t_n$. Then for each $X=\{\wm^{j_1}, \wm^{j_2}, \dots, \wm^{j_p}\}\in \PC(\CC)$ the unique way to write $X$ as a ``sum" (i.e., symmetric difference) of elements $s_j$ is $X=\{\wm^{j_1}\}\triangle \{\wm^{j_2}\}\triangle \dots \triangle \{\wm^{j_p}\}$. Thus the condition (N1) holds for $\PC(\CC)$. Also for each $i\in[n]$ and $j\in [r]$ we have $\mathrm{Tk}_i^\CC\cap \{\wm^j\}\in \{\emptyset, \{\wm^j\}\}$, so that the condition (N2) holds for $\PC(\CC)$. Finally, let $\{\wm^j\}$, $\{\wm^k\}$ be two distinct elements of $S$. Let $i$ be a coordinate on which one of $\wm^j$, $\wm^k$ has $0$ and the other one $1$. Then exactly one of $\mathrm{Tk}_i^\CC\cap \{\wm^j\}$, $\mathrm{Tk}_i^\CC\cap \{\wm^k\}$ is $\emptyset$. Thus the condition (N3) holds for $\PC(\CC)$.

Conversely, suppose that we have a non-zero commutative ring $R$ which has a subset $S$ and a sequence of its elements $T$ satisfying the conditions (N1), (N2), and (N3).

\smallskip\noindent
\underbar{Claim 1.} No element of $S$ is equal to $0$.

\noindent
\underbar{Proof.} Suppose $0\in S$. If $S=\{0\}$, then, by (N1), $R=\{0\}$, a contradiction. Suppose $S\ne\{0\}$ and let $s\ne 0$ be a non-zero element of $S$. Then $s$ and $s+0$ are two different ways to write an element of $R$ as a sum of distinct elements of $S$, a contradiction. \underbar{Claim 1 is proved.}

\smallskip\noindent
\underbar{Claim 2.} If $1\in S$, then $S=\{1\}$.

\noindent
\underbar{Proof.} Suppose $1\in S$ and $S\ne \{1\}$. Let $s\in S$, $s\ne 1$. Then, by (N3), there is a $t$ from the sequence $T$ such that exactly one $t1$, $ts$ is equal to $0$. If $t1=0$, then, by (N2) and (N3), $ts=s$. However, $t1=0$ implies $t=0$, hence $ts=0$. Hence $s=0$, contradicting Claim 1. The other option is that $ts=0$. Then, by (N2) and (N3), $t1=1$, hence $t=1$, hence $0=ts=s$, again contradicting Claim 1. \underbar{Claim 2 is proved.}

\smallskip\noindent
\underbar{Proof for the case $S=\{1\}$.} Suppose $S=\{1\}$. Then $R=\{0,1\}$. Hence each $t_i$ is either $0$ or $1$. We form a codeword $\wm=w_1w_2\dots w_n\in\F_2^n$ in the following way: if $t_i=0$, we put $w_i=0$, and if $t_i=1$, we put $w_i=1$. Let $\CC=\{\wm\}$. Then $\PC(\CC)=\{\emptyset, \{\wm\}=\CC\}$. The map $\phi:R\to \PC(\CC)$, defined by $\phi(0)=\emptyset$, $\phi(1)=\CC$, is a ring isomorphism. We also have $\phi(t_i)=\emptyset=\mathrm{Tk}_i^\CC$ if $t_i=0$, and $\phi(t_i)=\CC=\mathrm{Tk}_i^\CC$ if $t_i=1$. \underbar{The proof for the case $S=\{1\}$ is finished.}

\smallskip
From now on we assume that $1\notin S$. Equivalently, $|S|\ge 2$ (due to Claim 2 and the fact that $1$ is representable as a sum of distinct elements of $S$).

\smallskip\noindent
\underbar{Claim 3.} For any two distinct elements $s_j, s_k\in S$, $s_js_k=0$.

\noindent
\underbar{Proof.} Let $s_j, s_k$ be two distinct elements of $S$. By (N3) there is an element $t_i$ from $T$ such that exactly one of the elements $t_is_j$, $t_is_k$ is $0$. Say $t_is_j=0$. Then, by (N2) and (N3), $t_is_k=s_k$. Now $t_is_js_k=(t_is_j)s_k=0s_k=0$, and $t_is_js_k=s_j(t_is_k)=s_js_k$. Hence $s_js_k=0$. \underbar{Claim 3 is proved.}

\smallskip\noindent
\underbar{Claim 4.} For any element $s_j\in S$, $s_js_j=s_j$.

\noindent
\underbar{Proof.} Let $1=s_{j_1}+\dots+s_{j_p}$ $(p\ge 2)$ be the unique representation of $1$ as a sum of distinct elements of $S$. If $p<|S|$, then there is an $s_j\in S$ not participating in the representation of $1$. Multiplying the representation of $1$ by $s_j$ and using Claim 3, we get $s_j=0$, contradicting to Claim 1. Hence $p=|S|$, i.e., $1=s_1+s_2+\dots+s_r$. Now for any $j\in [r]$, when we multiply this representation of $1$ by $s_j$, we get (using Claim 3) that $s_j=s_js_j$. \underbar{Claim 4 is proved.}

\smallskip\noindent
\underbar{Claim 5.} For any element $s_j\in S$, $s_j+s_j=0$.

\noindent
\underbar{Proof.} Note that $s_j+s_j\ne s_j$, otherwise, by cancellation, $s_j=0$, contradicting Claim 1. Suppose that $s_j+s_j=s_j+s_{j_1}+\dots+s_{j_p}$ with $p\ge 1$ and all $s_{j_\mu}$ $(\mu=1,2,\dots, p)$ different than $s_j$. Cancelling $s_j$ we get $s_j=s_{j_1}+\dots+s_{j_p}$, contradicting to (N1). Suppose now that $s_j+s_j=s_{j_1}+\dots+s_{j_p}$ with $p\ge 1$ and all $s_{j_\mu}$ $(\mu=1,2,\dots, p)$ different than $s_j$. If we multiply this equality by $s_{j_1}$ and use the claims 3 and 4, we get $s_{j_1}=0$, contradicting Claim 1. The only remaining option is $s_j+s_j=0$. \underbar{Claim 5 is proved.} 

\smallskip\noindent
\underline{Proof for the case $S\ne\{1\}$ (i.e., $|S|\ge 2$).} For every element $s\in S$ we construct  a word $\wm=w_1w_2\dots w_n\in\F_2^n$ in the following way: for $i=1,2,\dots, n$, if $t_is=0$ we put $w_i=0$, otherwise (if $t_is=s$) we put $w_i=1$. In that way we get $r$ words $\wm^1, \wm^2,\dots, \wm^r$ from $\F_2^n$, corresponding, respectively, to $s_1, s_2, \dots, s_r$. Let $\CC=\{\wm^1, \wm^2,\dots, \wm^r\}$.  
For every $x\in R$, if $x=s_{j_1}+\dots+s_{j_p}$ is the unique representation of $x$ as a sum of distinct elements of $S$, we define
\begin{align*}
S(x)&=\{s_{j_1}, s_{j_2},\dots, s_{j_p}\}\subseteq R,\\
W(x)&=\{\wm^{j_1}, \wm^{j_2},\dots, \wm^{j_p}\}\subseteq \CC.
\end{align*}
Note that for any $x,y\in R$ we have
\[S(x+y)=S(x)\triangle S(y)\]
due to Claim 5, and
\[S(xy)=S(x)\cap S(y)\]
due to the claims 3 and 4. Hence
\begin{align}
W(x+y)&=W(x)\triangle W(y),\label{eq1}\\
W(xy)&=W(x)\cap W(y). \label{eq2}
\end{align}
Note also that if $x=0$, $S(x)=\emptyset$, hence
\begin{equation*}
W(0)=\emptyset,
\end{equation*}
and if $x=1$, $S(x)=S$ by the proof of Claim 4, hence
\begin{equation}
W(1)=\CC.\label{eq3}
\end{equation}
Now we define a map $\phi:R\to\CC$ as $\phi(x)=W(x)$ for any $x\in R$. The relations (\ref{eq1}), (\ref{eq2}), and (\ref{eq3}) show that $\phi$ is a ring homomorphism. Also
\begin{equation*}
\phi(s_j)=\{\wm^j\} \text{ for every $j\in [n]$}.
\end{equation*}
It remains to find $\phi(t_i)$ for each $i\in[n]$. Fix an $i\in[n]$. Let $t_i=s_{j_1}+s_{j_2}+\dots+s_{j_p}$ $(p\ge 0)$ be the unique representation of the element $t_i$ as a sum of distinct element of $S$. Multiplying this representation by $s_{j_\mu}$ $(\mu\in[p])$ and using the claims 3 and 4 we conclude that 
\begin{equation}
t_is_{j_\mu}=s_{j_\mu} \; (\mu=1,2,\dots, p).\label{eq4}
\end{equation}
We claim that 
\begin{equation}
t_is_j=0 \text{ for any $s_j\in S\setminus\{s_{j_1},\dots, s_{j_p}\}$.}\label{eq5}
\end{equation}
Suppose to the contrary, i.e., $t_is_j=s_j$ for some $s_j\in S\setminus\{s_{j_1},\dots, s_{j_p}\}$. Then, by Claim 3,
$s_j=t_is_j=(s_{j_1}+s_{j_2}+\dots+s_{j_p})s_j=0$,
a contradiction. Thus 
\[\phi(t_i)=\{\wm^{j_1},\dots, \wm^{j_p}\},\]
which is precisely the set of all the words from $\CC$ that have the $i$-th coordinate equal to $1$ (due to (\ref{eq4}), (\ref{eq5}), and the way the code $\CC$ is constructed). Thus
\[\phi(t_i)=\mathrm{Tk}_i^\CC \,\text{ for all $i\in[n]$.}\qedhere\]
\end{proof}

\bigskip
Next we give an intrinsic characterization of homomorphisms bet\-ween neural rings.

\begin{theorem}\label{thm_char_hom}
Let $\CC, \DC$ be two codes and let $\phi: (\PC(\DC),\triangle,\cap)\to (\PC(\CC),\triangle,\cap)$ be a map between these two rings. The following are equi\-va\-lent:

(i) $\phi$ is a ring homomorphism;

(ii) $\phi$ satisfies the following three conditions:

\quad\quad (H1) $\phi(\{\vm^1\})\cap \phi(\{\vm^2\})=\emptyset$ for any $\vm^1, \vm^2\in\DC$;

\quad\quad (H2) $(\forall B\subseteq \DC) \; \phi(B)=\cup_{\vm\in B}\, \phi(\{\vm\})$;

\quad\quad (H3) $\phi(\DC)=\CC$;

(iii) $\phi=q^{-1}$ for some map $q:\CC\to\DC$. 
\end{theorem}

\begin{proof}
(i) $\Rightarrow$ (ii): Suppose that $\phi: (\PC(\DC),\triangle,\cap)\to (\PC(\CC),\triangle,\cap)$ is a ring homomorphism. Then for two distinct elements $\vm^1, \vm^2$ of $\DC$ we have:
\begin{align*}
\emptyset &=\phi(\emptyset)\\
                 &=\phi(\{\vm^1\}\cap \{\vm^2\})\\
                 &=\phi(\{\vm^1\})\cap \phi(\{\vm^2\}).
\end{align*}
Thus (H1) holds. 

We show (H2) by induction on $|B|$. For $|B|=1$ the statement is true. Suppose that (H2) holds when $|B|=k$ and suppose that $|B|=k+1$. Let $B=B'\cup \{\wm\}$, where $|B'|=k$. Then
\begin{align*}
\phi(B)&=\phi(B'\cup \{\wm\})\\
           &=\phi(B'\triangle \{\wm\})\\
           &=\phi(B')\triangle\, \phi(\{\wm\})\\
           &=(\cup_{\vm'\in B'} \phi(\{\vm'\}) \triangle\, \phi(\{\wm\})\\
           &=(\cup_{\vm'\in B'} \phi(\{\vm'\}) \cup \phi(\{\wm\}) \quad\quad\quad\quad \text{(by (H1))}\\
           &=\cup_{\vm\in B} \phi(\{\vm\}).
\end{align*}
Thus (H2) holds.

Finally $\phi(\DC)=\CC$ as the identity element has to be mapped to the identity element. Thus (H3) holds.

(ii) $\Rightarrow$ (i): Suppose that  $\phi: (\PC(\DC),\triangle,\cap)\to (\PC(\CC),\triangle,\cap)$ is a map satisfying the conditions (H1), (H2), and (H3). Let $B_1, B_2\in\PC(\DC)$. We have:
\begin{align*}
\phi(B_1\triangle B_2)&=\cup_{\vm\in B_1\triangle B_2}\, \phi(\{\vm\})\\
                                   &=\cup_{\vm\in B_1}\phi(\{\vm\})\, \triangle \cup_{\vm\in B_2}\phi(\{\vm\})\\
                                   &=\phi(B_1)\triangle\, \phi(B_2).
\end{align*}
 We used here the conditions (H1) nd (H2). In the same way we get $h(B_1\cap B_2)=\phi(B_1)\cap \phi(B_2)$. Finally the condition $\phi(\DC)=\CC$ is postulated. Thus $\phi$ is a ring homomorphism.

(ii) $\Rightarrow$ (iii): Suppose (ii) holds. Let $\um\in\CC$. Then $\um\in\phi(\DC)$ by (H3), hence $\um\in\cup_{\vm\in\DC} \phi(\vm)$ by (H2), hence $\um\in\phi(\vm)$ for some $\vm\in\DC$. Such a $\vm$ is unique by (H1). Thus for every $\um\in\CC$ there is a unique $\vm(\um)\in\DC$ such that $\um\in\phi(\vm(\um))$. Define $q:\CC\to\DC$ by putting $q(\um)=\vm(\um)$.
Now for every $\vm\in\DC$ we have
\begin{align*}
q^{-1}(\{\vm\})&=\{\um\;:\;\vm(\um)=\vm\}\\
                          &=\{\um\;:\;\um\in\phi(\{\vm\})\}\\
                          &=\phi(\{\vm\}).
\end{align*}
Hence, by (H2),
\begin{align*}
\phi(B)&=\cup_{\vm\in B}\, \phi(\{\vm\})\\
           &=\cup_{\vm\in B}\, q^{-1}(\vm)\\
           &=q^{-1}(B)
\end{align*}
for every $B\subseteq \DC$. Thus $\phi=q^{-1}$.

(iii) $\Rightarrow$ (ii): Clear.
\end{proof}

\begin{proposition}[{\cite[Theorem 1.1]{cy}}]\label{correspondence_thm}
Let $\CC, \DC$ be two codes. There is a bijective correspondence between the set of maps $q:\CC\to \DC$ and the set of ring homomorphisms $\phi:(\PC(\DC),\triangle,\cap)\to (\PC(\CC),\triangle,\cap)$. It is given by associating to each map $q:\CC\to \DC$ the homomorphism $q^{-1}:\PC(\DC)\to \PC(\CC)$ and, conversely, by associating to each ring homomorphism $\phi:\PC(\DC)\to \PC(\CC)$ the unique map $q_\phi:\CC\to\DC$ such that $\phi=q_\phi^{-1}$.
\end{proposition}

\begin{proof}
For any map $q:\CC\to\DC$, the map $q^{-1}:(\PC(\DC),\triangle,\cap)\to(\PC(\CC),\triangle,\cap)$ is a ring homomorphism by Theorem \ref{thm_char_hom}, and if $q\ne q'$, then $q^{-1}\ne q'^{-1}$. Also, if $\phi:(\PC(\DC),\triangle,\cap)\to (\PC(\CC),\triangle,\cap)$ is a ring homomorphism, then (by the proof of Theorem \ref{thm_char_hom}) 
the map $q_\phi:\CC\to\DC$, defined by putting $q_\phi(\um)=\vm$ if $\um\in\phi(\{\vm\})$,
satisfies $\phi=q_\phi^{-1}$. It is a unique map from $\CC$ to $\DC$ with that property as two different maps from $\CC$ to $\DC$ give rise to different inverse maps from $\PC(\DC)$ to $\PC(\CC)$. Also, if $\phi\ne\phi'$, then $q_\phi\ne q_{\phi'}$.
\end{proof}

\begin{definition}\label{def_code_map}
A {\it code map} is any map $q:\CC\to\DC$, where $\CC, \DC$ are two codes.
\end{definition}

\begin{definition}\label{def_neural_ring_hom}
A {\it neural ring homomorphism} is any ring homomorphism $\phi:(\PC(\DC),\triangle,\cap)\to (\PC(\CC),\triangle,\cap)$,  where $\CC, \DC$ are two codes. 
\end{definition}

\begin{definition}\label{assoc_maps_homs}
We say that a code map $q:\CC\to\DC$ and the neural ring homomorphism $q^{-1}:\PC(\DC)\to\PC(\CC)$ are {\it associated} to each other. Equivalently, we say that a neural ring homomorphism $\phi:(\PC(\DC),\triangle,\cap)\to (\PC(\CC),\triangle,\cap)$ and the unique code map $q_\phi:\CC\to\DC$ such that $\phi=q_\phi^{-1}$, are {\it associated} to each other.
\end{definition}


\section{Monomial and linear monomial maps and homomorphisms}

\begin{definition}\label{def_Code}
We denote by $\mathrm{Codes}$ the set of all neural codes $\CC\subseteq \F_2^n$ of all lengths $n\ge 0$. The set $\mathrm{Codes}$, together with code maps as morphisms, forms a small category, which we denote by $\mathrm{\bf Code}$.
\end{definition}

\begin{definition}
We denote
\[\mathrm{NRings}=\{(\PC(\CC), \triangle, \cap)\mid \CC\in\mathrm{Codes}\}\]
and call this set the {\it set of all neural rings}.
The set $\mathrm{NRings}$, together with ring homomorphisms as morphisms, forms a small category, which we denote by $\mathrm{\bf NRing}$.
\end{definition}

\begin{theorem}\label{thm_Code_NRing}
Consider the categories $\mathrm{\bf Code}$ and $\mathrm{\bf NRing}$. If to each code $\CC\in\mathrm{Codes}$ we associate its neural ring $F(\CC)=(\PC(\CC),\triangle, \cap)$ and to each code map $q:\CC\to\DC$ the homomorphism of neural rings $F(q)=q^{-1}:\PC(\DC)\to \PC(\CC)$, then in this way we obtain a contravariant functor $F:\mathrm{\bf Code}\to\mathrm{\bf NRing}$, which is an isomorphism of these categories.
\end{theorem}
\begin{proof}
It is easy to verify that $F$ is a functor between these categories. The fact that $F$ is an isomorphism follows from Proposition \ref{correspondence_thm}.
\end{proof}

\begin{definition}[{\cite[Section 1.5]{cy}}]
The following code maps are called {\it basic linear monomial maps}:

(1) $\mathrm{acz}_\CC:\CC\to \CC'=\mathrm{acz}_\CC(\CC)$, ``adding constant zero", defined by $\wm\mapsto \wm 0$ for all $\wm\in\CC$, where $\CC\in\mathrm{Codes}$;

(2) $\mathrm{aco}_\CC:\CC\to \CC'=\mathrm{aco}_\CC(\CC)$, ``adding constant one", defined by $\wm\mapsto \wm 1$ for all $\wm\in\CC$, where $\CC\in\mathrm{Codes}$;

(3) $\mathrm{del}_{\CC,i}:\CC\to \CC'=\mathrm{del}_{\CC,i}(\CC)$, ``deleting the $i$-th neuron", defined by $\wm\mapsto w_1\dots\widehat{w_i}\dots w_n$ for all $\wm=w_1w_2\dots w_n\in\CC$, where $\CC\in\mathrm{Codes}$ and $i\in [n]$; here the notation $\widehat{w_i}$ means that the $i$-th component of $\wm$ is omitted;

(4) $\mathrm{rep}_{\CC,i}:\CC\to \CC'=\mathrm{rep}_{\CC,i}(\CC)$, ``repeating the $i$-th neuron", defined by $\wm\mapsto \wm w_i$ for all $\wm=w_1w_2\dots w_n\in\CC$, where $\CC\in\mathrm{Codes}$ and $i\in [n]$;

(5) $\mathrm{per}_{\CC,\sigma}:\CC\to \CC'=\mathrm{per}_{\CC,\sigma}(\CC)$, ``permuting the indices", defined by $\wm\mapsto w_{\sigma(1)}w_{\sigma(2)}\dots w_{\sigma(n)}$ for all $\wm=w_1w_2\dots w_n\in\CC$, where $\CC\in\mathrm{Codes}$ and $\sigma\in S_n$;

(6) $\mathrm{inj}_{\CC', \CC}:\CC\to\CC'$, ``injecting the code into a bigger code", defined by $\wm\mapsto \wm$ for all $\wm\in\CC$, where $\CC,\CC'\in\mathrm{Codes}$ with $\CC\subseteq \CC'$.
\end{definition}

We extend the previous definition and introduce the notion of {\it basic monomial maps} by including all the basic linear monomial maps and adding one new code map that we call {\it adding a trunk neuron}.

\begin{definition}
The code maps (1) - (6) and the following code map (7) are called {\it basic monomial maps}: 

(7) $\mathrm{atn}_{\CC,\alpha}:\CC\to \CC'=\mathrm{atn}_{\CC,\alpha}(\CC)$, ``adding a trunk neuron", defined by 
\[\mathrm{atn}_{\CC,\alpha}(\wm)=
                               \begin{cases}
                                    \wm 1, & \text{if $\wm\in \mathrm{Tk}_\alpha^\CC$}\\
                                    \wm 0, & \text{if $\wm\notin \mathrm{Tk}_\alpha^\CC$,}  
                               \end{cases}
\]
\end{definition}for all $\wm\in\CC$, where $\CC\in\mathrm{Codes}$ and $\alpha\subseteq [n]$.

\begin{remark}\label{rmk_tms}
We will usually write $\mathrm{atn}_{\CC, i}$ instead of $\mathrm{atn}_{\CC,\{i\}}$. Note that
\[\mathrm{atn}_{\CC,i}=\textrm{rep}_{\CC,i}\]
and
\[\mathrm{atn}_{\CC,\emptyset}=\mathrm{aco}_\CC.\]
Also, if $\mathbf{1}=11\dots 1\notin\CC$, then
\[\mathrm{atn}_{\CC,[n]}=\mathrm{acz}_\CC.\]
\end{remark}

\begin{definition}[{\cite[Section 1.5]{cy}}]
A map $q:\CC\to \DC$, where $\CC, \DC\in\mathrm{Codes}$, is called a {\it linear monomial map} if the inverse image under $q$ of every simple trunk of $\DC$ is either a simple trunk of $\CC$, or the empty set, or $\CC$.
\end{definition}

\begin{proposition}\label{basic_linear_monomial}

(a) Every basic linear monomial map is a linear monomial map.

(b) A composition of two linear monomial maps is a linear monomial map.

(c) For any code $\CC$ the identity map $\mathrm{Id}_\CC:\CC\to\CC$ is a linear monomial map.
\end{proposition}

\begin{proof}
 (a) Basic linear monomial maps are linear monomial maps by \cite{cy}.

(b), (c): Easy.
\end{proof}

\begin{theorem}[{\cite[Theorem 1.4]{cy}}]\label{thm_lmm}
A map $q:\CC\to \DC$, where $\CC, \DC\in\mathrm{Codes}$, is a {\it linear monomial map} if and only if it is the a composition of finitely many basic linear monomial maps.
\end{theorem}

\begin{definition}[{\cite[Definition 2.6]{a}}]
A map $q:\CC\to \DC$, where $\CC, \DC\in\mathrm{Codes}$, is called a {\it monomial map} if the inverse image under $q$ of every simple trunk of $\DC$ is either a trunk of $\CC$, or the empty set.
\end{definition}

\begin{proposition}\label{basic_monomial}

(a) Every linear monomial map is a monomial map.

(b) Every basic monomial map is a monomial map.

(c) A composition of two monomial maps is a monomial map.

(d) For any code $\CC$ the identity map $\mathrm{Id}_\CC:\CC\to\CC$ is a monomial map.
\end{proposition}

\begin{proof}
(a) Follows from the definitions and the fact that $\CC$ is a trunk, namely $\CC=\mathrm{Tk}^\CC_\emptyset$.

(b) Basic linear monomial maps are linear monomial maps by \cite{cy}, hence monomial maps. Consider the map $f=\mathrm{atn}_{\CC,\alpha}:\CC\to \DC=\mathrm{atn}_{\CC,\alpha}(\CC)$, where $\CC$ is a code on $n$ neurons. We have $f^{-1}(\mathrm{Tk}_i^{\DC})=\mathrm{Tk}_i^\CC$ for $i=1,\dots, n$. Also $f^{-1}(\mathrm{Tk}_{n+1}^{\DC})=\mathrm{Tk}_\alpha^\CC$.

(c) and (d): Easy.
\end{proof}

We now extend Theorem \ref{thm_lmm} to the case of monomial maps. Our proof follows the proof of Theorem 1.4 from \cite{cy}.

\begin{theorem}\label{thm_mm}
A map $q:\CC\to \DC$, where $\CC, \DC\in\mathrm{Codes}$, is a {\it monomial map} if and only if it is the a composition of finitely many basic monomial maps.
\end{theorem}

\begin{proof}
$\Leftarrow$) Follows from Proposition \ref{basic_monomial}.

$\Rightarrow$) Let $\CC$ be a code of length $m$, $\DC$ a code of length $n$, and let $q:\CC\to\DC$ be a monomial map. We introduce the codes $\CC_0, \CC_1,\dots, \CC_n$ in the following way: we put $\CC_0=\CC$ and, for $i\in[n]$,
\[\CC_i=\{\um v_1v_2\dots v_i\;|\; \um\in\CC, \vm=v_1v_2\dots v_n=q(\um)\}.\]
We also introduce the code maps $q_i:\CC_{i-1}\to\CC_i$ $(i=1,2,\dots, n)$ in the following way:
\[q_i(\um v_1v_2\dots v_{i-1})=\um v_1v_2\dots v_{i-1}v_i,\]
where $\um\in\CC$ and $\vm=v_1v_2\dots v_n=q(\um)$. Since $q^{-1}(\mathrm{Tk}_i^{\DC})$ for $i=1,2,\dots, n$ is either $\mathrm{Tk}_\alpha^\CC$ (with $\alpha\subseteq [i]$), or $\emptyset$, or $\CC$, we have that $q_i$ is, respectively, either $\mathrm{atn}_{\CC_{i-1},\alpha}$, or $\mathrm{acz}_{\CC_{i-1}}$, or $\mathrm{aco}_{\CC_{i-1}}=\mathrm{atn}_{\CC_{i-1},\emptyset}$. We also introduce the code $\CC_{n+1}$ in the following way: 
\[\CC_{n+1}=\{\vm\um\;|\;\um\in\CC, \vm=q(\um)\}.\]
Let $\sigma\in S_{m+n}$ be the permutation defined by $\sigma(i)=i+n$ for $i=1,2,\dots, m$, and $\sigma(i)=i-m$ for $i=m+1, m+2, \dots, m+n$. Let $q_{n+1}:\CC_n\to\CC_{n+1}$ be defined as
\[q_{n+1}=\mathrm{per}_{\CC_n,\sigma}.\]
Now for $i=1,2,\dots, m$ we introduce the codes $\CC_{n+1+i}$ in the following way:
\[\CC_{n+1+i}=\{\vm u_1u_2\dots u_{m-i}\;|\; \um=u_1u_2\dots u_m\in\CC, \vm=q(\um)\}.\]
We also introduce the code maps $q_{n+1+i}:\CC_{n+i}\to \CC_{n+1+i}$ $(i=1,2,\dots, m)$ in the following way:
\[q_{n+1+i}=\mathrm{del}_{\CC_{n+i}, n+m+1-i}.\]
Finally, we denote $\CC_{n+m+2}=\DC$ and introduce the code map $q_{n+m+2}:\CC_{n+m+1}\to \CC_{n+m+2}$ as
\[q_{n+m+2}=\mathrm{inj}_{\CC_{n+m+2}, \CC_{n+m+1}}.\]
We have that
\[q=q_{n+m+2}\circ q_{n+m+1}\circ\dots\circ q_1\]
and each of the maps $q_1, q_2,\dots, q_{n+m+2}$ is a basic monomial map.
\end{proof}

\begin{proposition}
(a) (\cite{a}) The set $\mathrm{Codes}$, together with  monomial maps as morphisms, forms a small category (which we denote $\mathrm{\bf Code\_ m}$).

(b) The set $\mathrm{Codes}$, together with linear monomial maps as morp\-hisms, forms a small category (which we denote $\mathrm{\bf Code\_ lm}$).
\end{proposition}

\begin{proof}
The proof of (a) given in \cite{a} works for (b) in a similar way.
\end{proof}

\begin{definition}
Let $\CC, \DC$ be two neural codes. A homomorphism $\phi:\PC(\DC)\to \PC(\CC)$ is called a {\it monomial homomorphism} (resp. {\it linear monomial homomorphism}) if the associated code map $q_\phi:\CC\to\DC$ is a monomial map (resp. linear monomial map).
\end{definition}

\begin{proposition}
(a) (\cite{a}) The set $\mathrm{NRings}$, together with monomial homomorphisms as morphisms, forms a small category (which we denote $\mathrm{\bf NRing\_ m}$).

(b) The set $\mathrm{NRings}$, together with linear monomial homomorphisms as morphisms, forms a small category (which we denote $\mathrm{\bf NRing\_ lm}$).
\end{proposition}

\begin{proof}
The proof of (a) given in \cite{a} works for (b) in a similar way.
\end{proof}

\begin{theorem}
(a) (\cite{a})  Consider the categories $\mathrm{\bf Code\_ m}$ and $\mathrm{\bf NRing\_ m}$. If to each code $\CC\in \mathrm{Codes}$ we associate its neural ring $F(\CC)=(\PC(\CC),\triangle,\cap)$ and to each monomial map $q:\CC\to\DC$ the monomial homomorphism of neural rings $F(q)=q^{-1}:\PC(\DC)\to \PC(\CC)$, then in that way we obtain a contravariant functor $F:\mathrm{\bf Code\_ m}\to \mathrm{\bf NRing\_ m}$, which is an isomorphism of these categories.

(b) Consider the categories $\mathrm{\bf Code\_ lm}$ and $\mathrm{\bf NRing\_ lm}$. If to each code $\CC\in \mathrm{Codes}$ we associate its neural ring $F(\CC)=(\PC(\CC),\triangle,\cap)$ and to each linear monomial map $q:\CC\to\DC$ the linear monomial homomorphism of neural rings $F(q)=q^{-1}:\PC(\DC)\to \PC(\CC)$, then in that way we obtain a contravariant functor $F:\mathrm{\bf Code\_ lm}\to \mathrm{\bf NRing\_ lm}$, which is an isomorphism of these categories.
\end{theorem}

\begin{proof}
The proof of (a) given in \cite{a} works for (b) in a similar way.
\end{proof}

\begin{remark}
Let $\CC$, $\DC$ be neural codes on $m$ and $n$ neurons respectively. Under the ring isomorphism $A\to \overline{L_A}$ between $(\PC(\CC),\triangle,\cap)$ and $R_\CC$ and between $(\PC(\DC),\triangle,\cap)$ and $R_\DC$, the relations $q^{-1}(\mathrm{Tk}^\DC_i)\in\{\mathrm{Tk}^\CC_j\;:\;j\in [m]\}\cup \{\emptyset, \CC\}$ and $q^{-1}(\mathrm{Tk}^\DC_i)\in\{\mathrm{Tk}^\CC_\alpha\;:\; \alpha\subseteq [m]\}\cup\{\emptyset\}$ become 
$q^{-1}(x^\DC_i)\in \{x^\CC_j\;:\;j\in[m]\}\cup \{0,1\}$ and 
$q^{-1}(x^\DC_i)\in \{x^\CC_\alpha=\prod_{j\in\alpha} x^\CC_j\;:\;\alpha\subseteq [m]\}\cup \{0\}$, respectively, which justifies the names ``li\-near monomial map", ``monomial map" ``linear monomial morphism",  and ``monomial morphism" used in this section, as well as the names ``linear monomial isomorphism" and ``monomial isomorphism" that are going to be used in the next section. Here $x^\CC_j$ (resp. $x^\DC_i$) denotes $X_j+\mathcal{I}(\CC)$ in $R_\CC$ (resp. $X_i+\mathcal{I}(\DC)$ in $R_\DC$).
\end{remark}

\begin{summary}\label{summary_1}
In the categories $\mathrm{\bf Code}$,  $\mathrm{\bf Code\_m}$, and $\mathrm{\bf Code\_lm}$ the set of objects is the same, namely the set Codes of all neural codes. The morphisms are: code maps in $\mathrm{\bf Code}$,  monomial maps in $\mathrm{\bf Code\_m}$, and linear monomial maps in $\mathrm{\bf Code\_lm}$.

In the categories $\mathrm{\bf NRing}$, $\mathrm{\bf NRing\_m}$, and $\bf \mathrm{\bf NRing\_lm}$ the set of objects is the same, namely the set NRings of all neural rings. The morphisms are: neural ring homomorphisms in $\mathrm{\bf NRing}$, monomial homomorphisms in $\mathrm{\bf NRing\_m}$, and linear monomial homomorphisms in $\bf \mathrm{\bf NRing\_lm}$.

The categories $\mathrm{\bf Code}$ and $\mathrm{\bf NRing}$,  $\mathrm{\bf Code\_m}$ and $\mathrm{\bf NRing\_m}$, and $\mathrm{\bf Code\_lm}$ and $\mathrm{\bf NRing\_lm}$ are isomorphic.
\end{summary}


\section{Monomial and linear monomial isomorphisms}

\begin{definition}[{\cite[page 12]{j}}]\label{def_isom}
Let $\mathrm{\bf C}$ be a  category. A map $f:A\to B$ between two objects of $\mathrm{\bf C}$ is said to be an {\it isomorphism} if it is a bijective morphism such that $f^{-1}:B\to A$ is also a morphism.
\end{definition}

\begin{theorem}[{\cite[Lemma 2.3]{cy}}]
The isomorphisms in the category $\mathrm{\bf Code}$ are bijective maps between codes. The isomorphisms in the ca\-te\-gory $\mathrm{\bf NRing}$ are the maps $q^{-1}:\PC(\DC)\to \PC(\CC)$, where $\CC$ and $\DC$ are codes, such that $q:\CC\to\DC$ is a bijection.
\end{theorem}

\begin{proof}
For the category $\mathrm{\bf Code}$ the statement folllows Definition \ref{def_Code} and Definition \ref{def_isom}.
For the category $\mathrm{\bf NRing}$ the statement then follows from Theorem \ref{thm_Code_NRing}.
\end{proof}

\begin{theorem}[{\cite[Corollary 1.5]{cy}}]
The isomorphisms in the category $\mathrm{\bf Code\_lm}$ are precisely the maps $per_{\CC,\sigma}:\CC\to \CC'=per_{\CC,\sigma}(\CC)$ which permute the indices of neurons. The isomorphisms in the ca\-te\-gory $\mathrm{\bf NRing\_lm}$ are the maps $q^{-1}:\PC(\DC)\to \PC(\CC)$, where $\CC$ and $\DC$ are codes, such that the associated maps $q:\CC\to\DC$ are isomorphisms in the category $\mathrm{\bf Code\_lm}$.
\end{theorem}

\begin{definition}
We call the isomorphisms in the categories $\mathrm{\bf Code\_m}$ and $\mathrm{\bf NRing\_m}$ (resp. $\mathrm{\bf Code\_lm}$ and $\mathrm{\bf NRing\_lm}$) {\it monomial} (resp. {\it li\-near monomial}) {\it isomorphisms}.
\end{definition}

The next definition is inspired by \cite{a}.

\begin{definition}\label{def_vector}
Let $\CC\subseteq \F_2^m$ be a neural code. A finite array $S=[T_1,T_2,\dots, T_n]$, $n\ge 1$, of subsets $T_i$ of $\CC$ is called a {\it vector of subsets of $\CC$}.
\end{definition}

\begin{proposition}\label{}
Let $\CC\subseteq \F_2^m$ be a neural code and let $S=[T_1,T_2,\dots, T_n]$ be a vector of subsets of $\CC$. There is one and only one code map $q_S:\CC\to \F_2^n$ such that $q_S^{-1}(\mathrm{Tk}_i^{\F_2^n})=T_i$ for $i=1,2,\dots, n$.
\end{proposition}\label{prop_q_S}

\begin{proof}
We define $q_S$ in the following way: for every $\um\in \CC$ let $\vm=q_S(\um)$ be the element of $\F_2^n$ whose $i$-th component $v_i$ is equal $1$ if $\um\in T_i$ and $0$ if $\um\notin T_i$. It is clear that then $q_S^{-1}(\mathrm{Tk}_i^{\F_2^n})=T_i$ for $i=1,2,\dots, n$ and that $q_S$ is the only map from $\CC$ to $\F_2^n$ with this property.
\end{proof}

\begin{examples}
(1) Let $\CC=\{001, 110, 101, 111\}\subseteq \F_2^3$, $T_1=\{101, 001\}$, $T_2=\{101, 111\}$, $S=[T_1, T_2, T_2, T_1]$. Then $q_S:\CC\to \F_2^4$ is given by $q_S(001)=1001$, $q_S(110)=0000$, $q_S(101)=1111$, $q_S(111)=0110$.

\smallskip
(2) Let $\CC=\{001, 110, 101, 111\}$ and $S=[\CC,\emptyset, \emptyset, \CC, \CC]$. Then $q_S:\CC\to \F_2^5$ is the constant map $q_S(\um)=10011$ for everu $\um\in\CC$.

\smallskip
(3) Let $\CC=\{00, 10, 01\}\subseteq \F_2^2$ and $S=[\mathrm{Tk}_1^\CC, \mathrm{Tk}_2^\CC, \CC]$. Then $q_S=aco_\CC:\CC\to \F_2^3$.

\smallskip
(4) Let $\CC=\{00, 10, 01\}\subseteq F_2^2$ and $S=[\mathrm{Tk}_1^\CC, \mathrm{Tk}_2^\CC, \emptyset]$. Then $q_S=acz_\CC:\CC\to \F_2^3$.

\smallskip
(5) Let $\CC=\{000, 101, 011\}\subseteq\F_2^3$ and $S=[\mathrm{Tk}_!^\CC, \mathrm{Tk}_3^\CC]$. Then $q_S=del_{\CC,2}:\CC\to \F_2^2$.

\smallskip
(6) Let $\CC=\{100, 010, 001, 110\}\subseteq \F_2^3$ and $S=[\mathrm{Tk}^\CC_3, \mathrm{Tk}^\CC_2, \mathrm{Tk}^\CC_1]$. Then $q_S=per_{\CC,\sigma}:\CC\to \F_2^3$, where $\sigma=(13)\in S_3$.

\smallskip
(7) Let $\CC=\{100, 010, 001, 110\}\subseteq \F_2^3$ and $S=[\mathrm{Tk}^\CC_1, \mathrm{Tk}^\CC_2, \mathrm{Tk}^\CC_3,  \mathrm{Tk}^\CC_2]$. Then $q_S=rep_{\CC, 2}:\CC\to \F_2^4$.

\smallskip
(8) Let $\CC=\{100, 010, 001, 110\}\subseteq \F_2^3$ and $S=[\mathrm{Tk}^\CC_1, \mathrm{Tk}^\CC_2, \mathrm{Tk}^\CC_3]$. Then $q_S=inj_{\CC, \F_2^3}:\CC\to \F_2^3$.

\smallskip
(9) Let $\CC=\{100, 011, 101, 111\}\subseteq \F_2^3$ and $S=[\mathrm{Tk}^\CC_1, \mathrm{Tk}^\CC_2, \mathrm{Tk}^\CC_3, \mathrm{Tk}_{\{1,3\}}^\CC]$. Then $q_S=atn_{\CC, \{1,3\}}:\CC\to \F_2^4$.
\end{examples}

\begin{definition}\label{def_defining_vec}
Let $\CC\subseteq \F_2^m$ and $\DC\subseteq\F_2^n$ be two codes and $q:\CC\to\DC$ a code map. Then the vector $S_q=[q^{-1}(\mathrm{Tk}_1^\DC), \dots, q^{-1}(\mathrm{Tk}_n^\DC)]$ of subsets of $\CC$ is called the {\it vector associated to $q$} or the {\it defining vector of $q$}.
\end{definition}

Note that the maps $q$ and $q_{S_q}$ have the same domain and range, however the codomain of $q$ is $\DC$ and the codomain of $q_{S_q}$ is $\F_2^n$.

\begin{definition}\label{def_red_neuron}
Let $\CC\subseteq\F_2^m$ be a code, $\alpha\subseteq[m]$, $i\in[m]\setminus \{\alpha\}$. We say that the neuron $i$ is {\it redundant to $\alpha$} if it is not constant zero and $\mathrm{Tk}_i^\CC=\mathrm{Tk}_\alpha^\CC$. We say that the neuron $i\in [m]$ is {\it redundant} if it is redundant to some $\alpha\subseteq [m]$.
\end{definition}

\begin{proposition}\label{prop_red_neuron}
Let $\CC\subseteq \F_2^m$ be a neural code, $i\in [m]$, and consider the code map $q=del_{\CC,i}:\CC\to \DC=q(\CC)$. The following are equivalent:

(a) the neuron $i$ is redundant or constant zero;

(b) $q$ is a monomial isomorphism.
\end{proposition}

\begin{proof}
(a) $\Rightarrow$ (b): Suppose that $i$ is redundant to some $\alpha\subseteq [m]$. The map $q$ is a linear monomial map whose defining vector is
\[S_q=[\mathrm{Tk}^\CC_1, \dots, \mathrm{Tk}^\CC_{i-1}, \mathrm{Tk}^\CC_{i+1},\dots, \mathrm{Tk}^\CC_m].\]
If for $\um=u_1\dots u_m$ and $\vm=v_1\dots v_m$ from $\CC$ we have $u_1\dots \widehat{u_i}\dots u_m=v_1\dots \widehat{v_i}\dots v_m$, then $\um_\alpha=\vm_\alpha$ (since $i\notin\alpha$), hence $u_i=v_i$ (since $i$ is redundant to $\alpha$), hence $\um=\vm$.  Thus $q$ is injective, hence bijective. For the inverse map $q^{-1}:\DC\to\CC$ the defining vector is 
\[S_{q^{-1}}=[\mathrm{Tk}^\DC_1, \dots, \mathrm{Tk}^\DC_{i-1},\mathrm{Tk}^\DC_\alpha, \mathrm{Tk}^\DC_{i+1},\dots, \mathrm{Tk}^\DC_m].\]
Hence $q^{-1}$ is a monomial map, hence $q$ is a monomial isomorphism..  

When $i$ is constant zero, the proof goes along the same lines.

\smallskip
(b) $\Rightarrow$ (a): Suppose that $q=del_{\CC,i}:\CC\to\DC=q(\CC)$ is a monomial isomorphism. Since $q$ is bijective, we have the inverse map $q^{-1}:\DC\to\CC$. Its defining vector is  
\[S_{q^{-1}}=[\mathrm{Tk}^\DC_1, \dots, \mathrm{Tk}^\DC_{i-1}, T, \mathrm{Tk}^\DC_i,\dots, \mathrm{Tk}^\DC_{m-1}],\]
where $T\subseteq \DC$. Since $q^{-1}$ is a monomial map, $T=\mathrm{Tk}^\DC_\beta$ for some set $\beta=\{j_1, j_2, \dots, j_t\}\subseteq [m-1]$ with $j_1<j_2<\dots<j_t$, $t\ge 0$, or $T=\emptyset$. Hence if $T\ne\emptyset$, the elements of $\CC$ that have $1$ on the $i$-th coordinate are precisely the elements of the trunk $\mathrm{Tk}_\alpha^\CC$, where
$\alpha=\{j_1, j_2,\dots, j_{s-1}, j_s+1, j_{s+1}+1,\dots, j_t+1\}$, with $t\ge 0$, $0\le s\le t$, $j_{s-1}<i<j_s+1$. Thus $i$ is redundant to $\alpha$ or constant zero.
\end{proof}

\begin{theorem}\label{thm_char_mon_isom}
A code map $q:\CC\to\DC$ is a monomial isomorphism if and only if it is a composition of finitely many bijective code maps each of which is of one of the following types:

(i) permutation of indices of neurons;

(ii) addition of a redundant neuron;

(iii) deletion of a redundant neuron;

(iv) addition of a constant zero neuron;

(v) deletion of a constant zero neuron.
\end{theorem}

\begin{proof}
It is easy to see that the bijective code maps of types (i), (ii), and (iv) are monomial isomorphisms. Also it follows from Proposition \ref{prop_red_neuron} that bijective code maps of types (iii) and (v) are monomial isomorphisms. Hence a composition of finitely many bijective maps, each of one of the types (i)-(v), is a monomial isomorphism.

Conversely, let $q:\CC\to \DC$ be a monomial isomorphism, where $\CC\subseteq \F_2^m$ and $\DC\subseteq \F_2^n$. We decompose $q$ into a composition of finitely many code maps in the same way in which a code map $q$ was decomposed in the proof of Theorem \ref{thm_mm}. All the maps that appear in the decomposition from the proof of Theorem \ref{thm_mm} will appear here, in the decomposition of the monomial isomorphism $q$ that we are considering, except the last map, $q_{n+m+2}$, since here $q$ is a bijection. Moreover, all of the maps $q_1, q_2, \dots, q_{n+m+1}$ are bijections. The maps $q_i$, $1\le i\le n$, are of types (ii) or (iv), as in the proof of Theorem \ref{thm_mm}. The map $q_{n+1}$ is of type (i). Since the map $q^{-1}:\DC\to \CC$ is a monomial isomorphism, the map $\vm u_1u_2\dots u_{m-i}\mapsto \vm u_1u_2\dots u_{m-1}u_{m-i+1}$ from $\CC_{n+1+i}$ to $\CC_{n+i}$ is of type (ii) or (iv) (in the same way in which the maps $q_i$, $1\le i\le n$, were of types (ii) or (iv)), so that the map 
$q_{n+1+i}=del_{\CC_{n+i}, n+m+1-i}: \vm u_1u_2\dots u_{m-i}u_{m-i+1}\mapsto \vm u_1u_2\dots u_{m-i}$ 
from $\CC_{n+i}$ to $\CC_{n+1+i}$ is of type (iii) or (v) (for $1\le i\le m$). Thus $q$ is a composition of finitely many code maps that are of types (i)-(v).
\end{proof}

\begin{corollary}
The monomial isomorphisms in the category $\mathrm{\bf Nring\_m}$ are the maps $q^{-1}:\PC(\DC)\to \PC(\CC)$, where $\CC$ and $\DC$ are codes, such that the associated maps $q:\CC\to\DC$ are isomorphisms in the category $\mathrm{\bf Code\_m}$ (characterized by theorem \ref{thm_char_mon_isom}).
\end{corollary}

\begin{summary}\label{summary_2}
The isomorphisms in the category $\mathrm{\bf Code}$ (resp. $\mathrm{\bf Code\_m}$; $\mathrm{\bf Code\_lm}$) are the bijective code maps (resp. the maps that are finite compositions of code maps of the five types listed in Theorem \ref{thm_char_mon_isom}; the maps that permute the indices of neurons).

The isomorphisms in the categories $\mathrm{\bf NRing}$ (resp. $\mathrm{\bf NRing\_m}$; $\mathrm{\bf NRing\_lm}$)  are the maps $q^{-1}:\PC(\DC)\to \PC(\CC)$, where $\CC$ and $\DC$ are codes, such that the associated maps $q:\CC\to\DC$ are isomorphisms in the category $\mathrm{\bf Code}$ (resp. $\mathrm{\bf Code\_m}$; $\mathrm{\bf Code\_lm}$).
\end{summary}

\end{document}